\documentclass[11pt,reqno]{amsart}

\usepackage[T1]{fontenc}
\usepackage{lmodern}
\usepackage{microtype}
\usepackage{amsmath,amssymb,amsthm,mathtools}
\usepackage{enumitem}
\usepackage{xcolor}
\usepackage{hyperref}
\usepackage[nameinlink,capitalize,noabbrev]{cleveref}

\hypersetup{
  colorlinks=true,
  linkcolor=blue!60!black,
  citecolor=blue!60!black,
  urlcolor=blue!60!black
}

\numberwithin{equation}{section}

\newtheorem{theorem}{Theorem}[section]

\newtheorem{lemma}[theorem]{Lemma}
\newtheorem{corollary}[theorem]{Corollary}

\theoremstyle{definition}

\newtheorem{example}[theorem]{Example}
\theoremstyle{remark}
\newtheorem{remark}[theorem]{Remark}

\newcommand{\HH}{\mathcal H}
\newcommand{\GG}{\mathcal G}
\newcommand{\EE}{\mathcal E}
\newcommand{\ZZ}{\mathbb Z}
\newcommand{\supp}{\operatorname{supp}}
\newcommand{\Lip}{\operatorname{Lip}}
\newcommand{\LLY}{\mathrm{LLY}}
\newcommand{\Wass}{W_{1}}
\newcommand{\edgecon}{\lambda}
\newcommand{\binomset}[2]{\binom{#1}{#2}}

\title[Lin--Lu--Yau Curvature and Edge-Connectivity of Hypergraphs]
{Lin--Lu--Yau Curvature and Edge-Connectivity of Hypergraphs}

\author{Qing Xia}
\address{School of Mathematical Sciences\\
University of Science and Technology of China\\
96 Jinzhai Road\\
Hefei 230026, Anhui Province\\
China}
\email{xq0420@mail.ustc.edu.cn}
\date{\today}


\begin{document}

\begin{abstract}
We give a local geometric criterion ensuring that the edge-connectivity of a hypergraph equals its minimum incidence degree: every locally finite connected simple $r$-uniform linear hypergraph with $r\ge3$ and nonnegative Lin--Lu--Yau curvature has this property. Among the various extensions of Ollivier--Ricci curvature to hypergraphs, we work with the equal-edges random walk on hypergraphs \cite{CoupetteEtAl2023}. Moreover, we show that both uniformity and linearity are essential: if either assumption is removed, there exist positively curved hypergraphs for which the gap between minimum incidence degree and edge-connectivity is arbitrarily large.

For adjacent vertices in any locally finite simple hypergraph, the
Ollivier--Ricci idleness function is concave and piecewise affine with
at most three affine pieces, and is affine on $[1/2,1]$. In particular,
$\kappa_{\LLY}=2\kappa_{1/2}$, which we use to verify positive curvature
in the two families of examples.
\end{abstract}

\maketitle

\section{Introduction and main results}

For a connected hypergraph $\HH$ with at least two vertices, let
$\edgecon(\HH)$ denote its edge-connectivity and let $\delta(\HH)$
denote its minimum incidence degree. Deleting all hyperedges incident
with a vertex of minimum incidence degree gives
$\edgecon(\HH)\le\delta(\HH)$. A hypergraph attaining equality is
called \emph{maximally edge-connected}. Several sufficient conditions
for this equality are expressed in terms of the overall structure of
the hypergraph. Dankelmann and Meierling \cite{DankelmannMeierling2016}
gave a minimum-degree condition in terms of the number of vertices.
Zhao and Meng \cite{ZhaoMeng2018} established criteria based on
distances and girth.
Zhao, Tian, and Meng \cite{ZhaoTianMeng2020} extended several
degree-sequence criteria to uniform hypergraphs.
Liu, Luo, and Ou \cite{LiuLuoOu2026} obtained sufficient conditions
involving the number of hyperedges or the degree sequence for
linear $r$-uniform hypergraphs.
Using vertex transitivity, Burgess, Luther, and Pike
\cite{BurgessLutherPike2024} extended Mader's theorem to connected
linear $r$-uniform hypergraphs with $r\ge 3$.

The preceding criteria use global structural properties of a
hypergraph, such as its degree sequence, distance structure, or
vertex transitivity. Discrete Ricci curvature offers a complementary
approach. Ollivier's coarse Ricci curvature measures the contraction of one-step probability distributions in the $1$-Wasserstein metric \cite{Ollivier2009}.  Its graph-theoretic form has become a basic tool in discrete geometry, while the high-idleness normalization introduced by Lin, Lu, and Yau gives the Lin--Lu--Yau curvature \cite{LinLuYau2011}. In particular, lower bounds on the Lin--Lu--Yau curvature of
adjacent vertex pairs can yield information about global graph
properties such as diameter, eigenvalues, and connectivity. For example,
Huang, Liu, and Xia \cite{HuangLiuXia2024} used local matchings to
estimate this curvature and obtained sharp diameter and eigenvalue
bounds for amply regular graphs. The connection with edge-connectivity
was established by Chen, Liu, and You
\cite[Theorem~1.5]{ChenLiuYou2025}, who proved that positive
Lin--Lu--Yau curvature forces a connected graph to be maximally
edge-connected. Liu and Xia \cite[Theorems~1.2--1.4]{LiuXia2026}
showed that nonnegative curvature suffices for finite connected
graphs and classified the locally finite infinite exceptions.

This leads us to ask whether a suitable nonnegative discrete
Ricci curvature condition can guarantee maximal edge-connectivity
in a locally finite hypergraph. For general hypergraphs, several extensions of Ollivier--Ricci curvature have been studied, including multi-marginal transport curvature
\cite{AsoodehGaoEvans2018}, curvature for directed hypergraphs
\cite{EidiJost2020}, nonlinear-Laplacian constructions
\cite{IkedaEtAl2022}, and the ORCHID framework
\cite{CoupetteEtAl2023}.
The latter includes pairwise curvatures associated with several
random walks, as well as aggregation procedures that assign
curvature to an entire hyperedge. We work with the equal-edges (EE) random walk: from a vertex,
one first chooses an incident hyperedge uniformly and then
chooses a different vertex of that hyperedge uniformly.
Its lazy version is the EE probability measure in
\cite[Eq.~(6)]{CoupetteEtAl2023}.
This walk agrees with the one used by Tian and Zhao
\cite{TianZhao2025} in the unweighted undirected setting. Throughout this paper, curvature is assigned to pairs of vertices,
using the usual $1$-Wasserstein distance induced by the hyperpath
metric; no aggregation over a hyperedge is involved.
The precise definitions are given in \cref{sec:prelim}.
Our main result is the following.

\begin{theorem}[Edge-connectivity rigidity]
\label{thm:linear-main}
Let $\HH=(V,E)$ be a locally finite connected simple $r$-uniform linear hypergraph, where $r\ge 3$ and $E\neq\varnothing$.  If $\HH$ has nonnegative Lin--Lu--Yau curvature, then
\[
   \edgecon(\HH)=\delta(\HH).
\]
\end{theorem}

The proof is based on the $2$-section reduction, curvature estimates on edge cuts, and a combinatorial inequality for finite bipartite $r$-uniform linear hypergraphs. The latter provides the required bound on the cross-neighborhood of a suitable adjacent pair in the hypergraph induced by the edge cut.

The assumptions of uniformity and linearity in \cref{thm:linear-main} are both necessary. In each case, there exist finite connected positively curved hypergraphs for which the difference between the minimum incidence degree and the edge-connectivity is arbitrarily large.

\begin{theorem}[Uniform nonlinear examples]
\label{thm:nonlinear-main}
For every integer $r\ge 3$ and every integer $t\ge 2$, there exists a finite connected simple nonlinear $r$-uniform hypergraph $\HH_{r,t}$ with positive Lin--Lu--Yau curvature such that
\[
   \edgecon(\HH_{r,t})=\delta(\HH_{r,t})-t.
\]
\end{theorem}

\begin{theorem}[Linear nonuniform examples]
\label{thm:linear-nonuniform}
For every integer $t\ge 1$, there exists a finite connected simple linear hypergraph $\mathcal L_t$ with positive Lin--Lu--Yau curvature such that every hyperedge has cardinality either $2$ or $3$ and
\[
   \edgecon(\mathcal L_t)=\delta(\mathcal L_t)-t.
\]
\end{theorem}

Thus uniformity and linearity are independent structural requirements for this rigidity: if either assumption is removed, strictly positive Lin--Lu--Yau curvature is compatible with an arbitrarily large gap between minimum incidence degree and edge-connectivity.

To verify positive curvature in the two families above, we study the
dependence of Ollivier's coarse Ricci curvature \cite{Ollivier2009}
on the idleness parameter. For graphs, Bourne, Cushing, Liu,
M\"unch, and Peyerimhoff \cite{BourneEtAl2018} proved that the
idleness function is concave and piecewise affine with at most three
affine pieces, and determined sharp intervals of affinity near
both endpoints.

For the EE random walk on a general hypergraph, transition
probabilities need not be uniform over the neighbors of a vertex.
Indeed, incident hyperedges may have different cardinalities,
and an adjacent pair may belong to several hyperedges.
Thus this walk need not coincide with the simple random walk
on the unweighted $2$-section, so the corresponding graph results
cannot be applied through that identification alone.
We prove that every adjacent-pair idleness function has at most
three affine pieces and is affine on $[1/2,1]$.
In \cref{subsec:other-walks}, we explain why the same idleness
argument applies to the other two random walks in the ORCHID
framework and identify the settings in which their pairwise
curvatures coincide.

\begin{theorem}[Idleness structure]
\label{thm:idleness-main}
Let $\HH=(V,E)$ be a locally finite simple hypergraph, and let $x\sim_{\HH}y$.  Then the function
\[
   \alpha\longmapsto \kappa_\alpha^{\HH}(x,y)
\]
is concave and piecewise affine on $[0,1]$ with at most three affine pieces.  Moreover, it is affine on $[1/2,1]$.  Consequently,
\begin{equation}
\label{eq:limit-free-intro}
   \kappa_{\LLY}^{\HH}(x,y)
   =\frac{\kappa_\alpha^{\HH}(x,y)}{1-\alpha},
   \qquad \alpha\in[1/2,1),
\end{equation}
and, in particular,
\begin{equation}
\label{eq:half-to-lly}
   \kappa_{\LLY}^{\HH}(x,y)=2\kappa_{1/2}^{\HH}(x,y).
\end{equation}
\end{theorem}

The three-piece bound and the high-idleness formula for the
equal-edges walk were established in my earlier preprint
\cite{XiaIdleness2026}. In a preprint posted shortly thereafter,
Liu and Yang \cite[Theorem~1.1]{LiuYang2026} independently
proved the high-idleness formula in a more general
weighted-graph setting.

For an $r$-uniform linear hypergraph, each adjacent pair lies in a unique hyperedge and the lazy EE random walk coincides exactly with the simple lazy random walk on the $2$-section.  Consequently, the graph-theoretic endpoint intervals of linearity apply after replacing the vertex degree in the $2$-section by $(r-1)$ times the incidence degree.

\begin{corollary}
\label[corollary]{cor:uniform-linear-idleness}
Let $\HH=(V,E)$ be a locally finite $r$-uniform linear hypergraph,
and let $x\sim_{\HH}y$. Let $[\HH]_2$ denote the $2$-section of $\HH$. Then
\[
\kappa_\alpha^{\HH}(x,y)
=
\kappa_\alpha^{[\HH]_2}(x,y)
\qquad(\alpha\in[0,1]),
\]
and consequently
\[
\kappa_{\LLY}^{\HH}(x,y)
=
\kappa_{\LLY}^{[\HH]_2}(x,y).
\]
Moreover, $\alpha\mapsto\kappa_\alpha^{\HH}(x,y)$ is affine on
\[
\left[0,
\frac{1}{\operatorname{lcm}(d_x^{\HH},d_y^{\HH})(r-1)+1}
\right]
\quad\text{and}\quad
\left[
\frac{1}{\max\{d_x^{\HH},d_y^{\HH}\}(r-1)+1},1
\right].
\]
If $d_x^{\HH}=d_y^{\HH}$, the idleness function has at most two affine pieces.
\end{corollary}

The limit-free identity \eqref{eq:half-to-lly} reduces the verification
of positive Lin--Lu--Yau curvature in
\cref{thm:nonlinear-main,thm:linear-nonuniform}
to transport at the fixed idleness $\alpha=1/2$.

The paper is organized as follows.  \Cref{sec:prelim} fixes the hypergraph, transport, curvature, and cut notation.  \Cref{sec:idleness} proves
\cref{thm:idleness-main,cor:uniform-linear-idleness},
and discusses the relation with other random-walk curvature
models.  \Cref{sec:comb} establishes the combinatorial inequality used for cut hypergraphs.  \Cref{sec:connectivity} proves \cref{thm:linear-main}.  Finally, \cref{sec:sharpness} constructs the two families in \cref{thm:nonlinear-main,thm:linear-nonuniform}.

\section{Preliminaries}
\label{sec:prelim}

\subsection{Hypergraphs, metrics, and associated graphs}
Throughout the paper, a hypergraph is a pair $\HH=(V,E)$, where
$V$ is a nonempty vertex set and $E$ is a family of finite subsets of $V$. For a set $U$, we write $|U|$ for its cardinality. We assume $|e|\ge 2$ for every $e\in E$.  The hypergraph is
\emph{simple} if $E$ is an antichain under inclusion (i.e., for distinct
$e,f\in E$, neither $e\subseteq f$ nor $f\subseteq e$).  It is
\emph{locally finite} if every vertex belongs to finitely many
hyperedges, and it is \emph{finite} if both $V$ and $E$ are finite. Unless otherwise stated, all hypergraphs in this paper are simple and locally finite.

Two distinct vertices $x,y\in V$ are \emph{adjacent}, written
$x\sim_{\HH}y$, if there exists $e\in E$ with $\{x,y\}\subseteq e$. We call $y$ a \emph{neighbor} of $x$ in $\HH$.
The incidence degree of $x$ is
\[
   d_x^{\HH}:=\#\{e\in E:x\in e\},
\]
and the minimum incidence degree is
\[
   \delta(\HH):=\min_{x\in V}d_x^{\HH}.
\]
A vertex is \emph{isolated} if its incidence degree is zero, and a
\emph{leaf} if its incidence degree equals $1$.

The hypergraph is \emph{$r$-uniform} if $|e|=r$ for every $e\in E$.
It is \emph{linear} if
\[
   |e\cap f|\le 1
   \qquad\text{for all distinct }e,f\in E.
\]
Otherwise it is nonlinear.

A hyperpath from $u$ to $v$ is a finite sequence of hyperedges
$(e_1,\dots,e_\ell)$ such that
\[
 u\in e_1,\qquad v\in e_\ell,\qquad
 e_i\cap e_{i+1}\ne\varnothing
 \quad(1\le i<\ell).
\]
The hyperpath distance is
\[
 d_{\HH}(u,v):=
 \min\{\ell:\text{there is a hyperpath of length $\ell$ from $u$ to $v$}\},
\]
whenever such a hyperpath exists. If no hyperpath joins $u$ and $v$,
we set $d_{\HH}(u,v):=\infty$. We put $d_{\HH}(u,u)=0$.  A hypergraph is connected if
$d_{\HH}(u,v)<\infty$ for all $u,v\in V$.

The \emph{$2$-section} $[\HH]_2$ of $\HH$ is the simple graph with vertex set
$V$ in which $x$ and $y$ are adjacent precisely when
$x\sim_{\HH}y$. In $[\HH]_2$, we denote the degree, adjacency relation, and graph distance by $d^{[\HH]_2}_x$, $x\sim_{[\HH]_2} y$, and $d_{[\HH]_2}(x,y)$, respectively. Note that the graph distance in $[\HH]_2$ equals
$d_{\HH}$.  If $\HH$ is $r$-uniform and linear, then
\begin{equation}
\label{eq:2section-degree}
   d_x^{[\HH]_2}=(r-1)d_x^{\HH}.
\end{equation}

The \emph{incidence graph} $\mathcal{B}(\HH)$ of
$\HH$ is the bipartite graph
\[
\mathcal{B}(\HH) = (V \sqcup E,\; I),
\]
with bipartition $V$ and $E$, and edge set
\[
I = \{\, (v,e) \in V\times E : v\in e \,\}.
\]
Thus two nodes of $\mathcal{B}(\HH)$ are adjacent exactly when
one is a vertex of $\HH$ and the other is a hyperedge containing it. 

For any finite graph $G$, the \emph{circuit rank} (also called the
\emph{cyclomatic number}, \emph{cycle rank}, or \emph{nullity}) is denoted
by $\beta(G)$ and is given by
\begin{equation}\label{eq:circuit_rank_formula}
\beta(G)=|E(G)|-|V(G)|+c(G),
\end{equation}
where $c(G)$ is the number of connected components of $G$; see
\cite{Diestel2017}. Note that for a finite connected $r$-uniform hypergraph $\HH$, its incidence graph $\mathcal{B}(\HH)$ is connected and has
$|V|+|E|$ vertices and $r|E|$ edges. 
Hence, by circuit rank formula \eqref{eq:circuit_rank_formula}, we have \begin{align}\label{eq:circuit_rank}
\beta\bigl(\mathcal{B}(\HH)\bigr)
& = r|E|-(|V|+|E|)+c(\mathcal{B}(\HH))\\
& = r|E|-(|V|+|E|)+1.
\nonumber
\end{align}

\subsection{Optimal transport and hypergraph curvature}

In all transport arguments, we work within a single connected
component of $\HH$ and write $V$ for its vertex set. Let $\mu$ and $\nu$ be finitely supported nonnegative measures
on $V$ with equal total mass, whose supports lie in a common
connected component of $\HH$. For $x\in V$, let $\delta_x$ denote the Dirac probability measure concentrated at $x$. A transport plan from $\mu$ to $\nu$ is a map
$\pi:V\times V\to[0,\infty)$ satisfying
\[
 \sum_{v\in V}\pi(u,v)=\mu(u),\qquad
 \sum_{u\in V}\pi(u,v)=\nu(v).
\]
Define their $1$-Wasserstein transportation cost by
\begin{equation}
\label{eq:wasserstein}
   \Wass^{\HH}(\mu,\nu)
   :=
   \inf_{\pi}
   \sum_{u,v\in V}
   d_{\HH}(u,v)\pi(u,v).
\end{equation}
When $\mu$ and $\nu$ are probability measures, this is the usual
$1$-Wasserstein distance.
By Kantorovich duality \cite{Villani2003},
\begin{equation}
\label{eq:kantorovich}
 \Wass^{\HH}(\mu,\nu)
 =
 \sup_{f\in 1\text{-}\Lip(V)}
 \sum_{z\in V}f(z)(\mu(z)-\nu(z)),
\end{equation}
where $f$ is $1$-Lipschitz when
$|f(u)-f(v)|\le d_{\HH}(u,v)$ for all $u,v\in V$. A transport plan $\pi$ from $\mu$ to $\nu$ is called an \emph{optimal transport plan} if it attains the infimum in \eqref{eq:wasserstein}. A $1$-Lipschitz function $f$ is called an \emph{optimal Kantorovich potential} from $\mu$ to $\nu$ if it attains the supremum in \eqref{eq:kantorovich}.

For $\alpha\in[0,1]$ and a nonisolated vertex $x\in V$, the
$\alpha$-lazy equal-edges (EE) random walk assigns the probability
measure
\begin{equation}
\label{eq:lazy-measure}
 \mu_x^{\alpha,\HH}(z)
 =
 \begin{cases}
 \alpha,
   &z=x,\\[2mm]
 \displaystyle
 \frac{1-\alpha}{d_x^{\HH}}
 \sum_{\substack{e\in E\\e\supseteq\{x,z\}}}
 \frac{1}{|e|-1},
   &z\sim_{\HH}x,\\[4mm]
 0,
   &\text{otherwise}.
 \end{cases}
\end{equation}
This is the EE measure of \cite[Eq.~(6)]{CoupetteEtAl2023},
under our standing assumption that every hyperedge has
cardinality at least two.
It also agrees with the measure used in \cite{TianZhao2025}
when all hyperedge weights are one.
We omit an EE superscript from the notation throughout.
In particular, $\mu_x^{1,\HH}=\delta_x$. 

For distinct vertices $x,y$ in the same connected component,
define their $\alpha$-Ollivier--Ricci curvature by
\begin{equation}
\label{eq:alpha-curvature}
 \kappa_\alpha^{\HH}(x,y)
 :=
 1-\frac{\Wass^{\HH}(\mu_x^{\alpha,\HH},\mu_y^{\alpha,\HH})}
          {d_{\HH}(x,y)}.
\end{equation}
For adjacent vertices $x\sim_{\HH}y$, their Lin--Lu--Yau
curvature is defined by
\begin{equation}
    \kappa_{\LLY}^{\HH}(x,y)
 :=
 \lim_{\alpha\uparrow 1}
 \frac{\kappa_\alpha^{\HH}(x,y)}{1-\alpha}.
\end{equation}
A locally finite hypergraph is said to have positive (resp., nonnegative) Lin--Lu--Yau curvature if $\kappa_{\LLY}^{\HH}(x,y)$ is positive for any $x,y\in V$ with $x\sim_{\HH} y$ (resp., nonnegative).

\subsection{Edge cuts and auxiliary bipartite structures}\label{sec:edge-cuts}

For a connected hypergraph $\HH=(V,E)$, let $\varnothing\ne X\subsetneq V$ and put $Y:=V\setminus X$.
The \emph{edge boundary} of $X$ (equivalently, of the bipartition
$V=X\sqcup Y$) is
\begin{equation}
\label{eq:edge-boundary}
 \partial_{\HH}(X)
 :=\{e\in E:e\cap X\ne\varnothing
                 \text{ and }e\cap Y\ne\varnothing\}.
\end{equation}
Clearly,
\[
   \partial_{\HH}(X)=\partial_{\HH}(Y).
\]
A set $F\subseteq E$ is called an \emph{edge cut} if there exists a
nonempty proper subset $X\subsetneq V$ such that
\[
   F=\partial_{\HH}(X)=\partial_{\HH}(V\setminus X).
\]
Thus an edge cut is exactly the common boundary of the two sides of
a nontrivial vertex bipartition. 

For a connected hypergraph $\HH=(V,E)$ with $|V|\ge2$, its edge-connectivity is
\begin{equation}
\label{eq:edge-connectivity-boundary}
   \edgecon(\HH)
   :=\min_{\varnothing\ne X\subsetneq V}
     |\partial_{\HH}(X)|.
\end{equation}
For the one-vertex hypergraph, we set $\edgecon(\HH):=0$. 

For $F\subseteq E$, write
\[
   \HH-F:=(V,E\setminus F).
\]

If $F=\partial_{\HH}(X)$, then $\HH-F$ contains no hyperedge meeting
both $X$ and $V\setminus X$, and hence is disconnected.  Conversely,
if a set $R\subseteq E$ satisfies that $\HH-R$ is disconnected and
$X$ is the union of one or more, but not all, components of
$\HH-R$, then
\[
   \partial_{\HH}(X)\subseteq R.
\]
Equality holds whenever $R$ is inclusion-minimal among the sets of
hyperedges whose deletion disconnects $\HH$.  This observation may
be used to pass between boundary language and deletion arguments,
but the definition throughout this paper is
\cref{eq:edge-boundary,eq:edge-connectivity-boundary}.

As usual,
\begin{equation}
\label{eq:trivial-upper}
   \edgecon(\HH)\le \delta(\HH).
\end{equation}
Indeed, the one-vertex case is immediate. If $|V|\ge2$, then for
every vertex $x$,
\[
   \partial_{\HH}(\{x\})=\{e\in E:x\in e\},
\]
whose cardinality is $d_x^{\HH}$.

For any $F\subseteq E$, the \emph{edge-induced subhypergraph} of $F$ is
\[
   \HH_F:=\bigl(V(F),F\bigr),
   \qquad
   V(F):=\bigcup_{e\in F}e.
\]
When $F$ is an edge cut of $\HH$, we call $\HH_F$ the \emph{cut hypergraph} associated with $F$.

We call a hypergraph $\GG=(U,\EE)$ \emph{bipartite with respect to}
$U=X\sqcup Y$ if every hyperedge meets both $X$ and $Y$. Such a hypergraph is a \emph{bipartite
star} if, after possibly interchanging $X$ and $Y$, there is a
vertex $x_0\in X$ such that
\[
   X=\{x_0\}
   \qquad\text{and}\qquad
   x_0\in e\quad\text{for every }e\in\EE.
\]
We call such $x_0$ the center of the bipartite star. Set $\varnothing\ne X\subsetneq V$. Let $F=\partial_{\HH}(X)$ be an edge cut and put $Y=V\setminus X$. Thus $\HH_F$ is bipartite with respect to
$(X\cap V(F))\sqcup (Y\cap V(F))$. This is a
weak bipartiteness condition: a hyperedge may contain more than one
vertex from either side. 

Let $\HH=(V,E)$ be bipartite with respect to
$V=X\sqcup Y$. For adjacent $x\in X$ and $y\in Y$, define
\[
S_{xy}^{\HH}=S_{yx}^{\HH}
:=
\{w\in Y\setminus\{y\}: w\sim_{\HH}x\}
\cup
\{w\in X\setminus\{x\}: w\sim_{\HH}y\}.
\]
Define its \emph{cross-neighbor graph}
$B(\HH)$ to be the bipartite graph with parts $X,Y$ and edge set
\[
 E(B(\HH))
 :=
 \{x\sim_{B(\HH)} y:x\in X,\ y\in Y,\ x\sim_{\HH}y\}.
\]
For any $(x,y)\in X\times Y$ with $x\sim_{B(\HH)} y$, we have
\begin{equation}
\label{eq:sxy-crossgraph}
 |S_{xy}^{B(\HH)}|=d_x^{B(\HH)}+d_y^{B(\HH)}-2,
\end{equation}
where 
\[
S_{xy}^{B(\HH)}:=\{w\in Y\setminus\{y\}: w\sim_{B(\HH)} x\}\cup\{w\in X\setminus\{x\}: w\sim_{B(\HH)} y\}.
\]
Moreover,
\begin{equation}
\label{eq:S_xyHH_S_xyB(HH)}
 S_{xy}^{\HH}=S_{xy}^{B(\HH)}.
\end{equation}
\section{Idleness functions and the Lin--Lu--Yau limit}
\label{sec:idleness}

\subsection{Transport lemmas}

We record several transport lemmas used in the proof of
\cref{thm:idleness-main}. The first, second, and fourth are direct
adaptations of \cite[Lemmas~3.1, 3.3, and~4.2]{BourneEtAl2018}
to the hypergraph metric; their proofs use only the underlying
discrete metric structure and are omitted. We also use the following
diagonal-saturation statement, which strengthens the form needed in
\cite[Lemma~4.1]{BourneEtAl2018}.

\begin{lemma}[cf.\ {\cite[Lemma~3.1]{BourneEtAl2018}}]
\label[lemma]{lem:complementary-slackness}
Let $\HH=(V, E)$ be a locally finite hypergraph. Let $x, y\in V$ with $x\sim_{\HH} y$. Let $\alpha\in[0,1]$. Let $\pi$ and $f$ be an optimal transport plan and an optimal Kantorovich potential transporting $\mu_{x}^{\alpha,\HH}$ to $\mu_{y}^{\alpha,\HH}$, respectively. Let $u, v\in V$ with $\pi(u, v)\neq 0$. Then
\[f(u)-f(v)=d_\HH(u, v).\]
\end{lemma}

\begin{lemma}[cf.\ {\cite[Lemma~3.3]{BourneEtAl2018}}]
\label[lemma]{lem:integer-potential}
Let $\HH=(V, E)$ be a locally finite hypergraph. Let $x\sim_{\HH}y$ and $\alpha\in[0,1]$.  There exists an optimal Kantorovich potential $f$ from $\mu_x^{\alpha,\HH}$ to $\mu_y^{\alpha,\HH}$ such that $f(V)\subseteq\ZZ$.
\end{lemma}

\begin{lemma}[Diagonal saturation]
\label[lemma]{lem:maximal-diagonal}
Let $\HH=(V,E)$ be a locally finite hypergraph, and let $\mu,\nu$
be finitely supported probability measures on $V$. Then there exists
an optimal transport plan $\pi$ from $\mu$ to $\nu$ such that
\[
   \pi(z,z)=\min\{\mu(z),\nu(z)\}
   \qquad(z\in V).
\]
\end{lemma}

\begin{proof}
Set
\[
\sigma(z):=\min\{\mu(z),\nu(z)\},
\qquad
\mu_0:=\mu-\sigma,
\qquad
\nu_0:=\nu-\sigma.
\]
By the common-mass subtraction property
\cite[Corollary~1.16]{Villani2003},
\[
\Wass^{\HH}(\mu,\nu)
=
\Wass^{\HH}(\mu_0,\nu_0).
\]
Let $\pi_0$ be an optimal transport plan from $\mu_0$ to $\nu_0$.
Since $\mu_0(z)\nu_0(z)=0$ for every $z\in V$, we have
$\pi_0(z,z)=0$. Hence
\[
\pi(u,v)
:=
\pi_0(u,v)
+
\sigma(u)\mathbf 1_{\{u=v\}}
\]
is an optimal transport plan from $\mu$ to $\nu$, and
\[
\pi(z,z)
=
\sigma(z)
=
\min\{\mu(z),\nu(z)\}
\qquad(z\in V).
\]
In particular,
\cite[Lemma~4.1]{BourneEtAl2018} follows as a special case.
\end{proof}

\begin{lemma}[cf.\ {\cite[Lemma~4.2]{BourneEtAl2018}}]
\label[lemma]{lem:common-potential-linearity}
Let $\HH=(V, E)$ be a locally finite hypergraph. Let $x\sim_{\HH}y$ and $0\le \alpha_1\le\alpha_2\le1$.  If a $1$-Lipschitz function $f$ is an optimal Kantorovich potential from $\mu_x^{\alpha_i,\HH}$ to $\mu_y^{\alpha_i,\HH}$ for both $i=1,2$, then
\[
   \alpha\longmapsto
   \Wass^{\HH}(\mu_x^{\alpha,\HH},\mu_y^{\alpha,\HH})
\]
is affine on $[\alpha_1,\alpha_2]$.
\end{lemma}

\subsection{Piecewise-affine structure and high-idleness linearity}

\begin{proof}[Proof of \cref{thm:idleness-main}]
Fix adjacent vertices $x\sim_{\HH}y$.  By Kantorovich duality and \cref{lem:integer-potential}, an optimal Kantorovich potential may be chosen integer-valued.  Since adding a constant to a Kantorovich potential does not change the dual objective, we normalize it by $f(y)=0$.  As $d_{\HH}(x,y)=1$, this forces
\[
   f(x)\in\{-1,0,1\}.
\]
For $j\in\{-1,0,1\}$, let
\[
\mathcal A_j
:=\{f:V\to\ZZ:\ f(y)=0,\ f(x)=j,\ f\text{ is }1\text{-Lipschitz}\}.
\]
Writing
\[
q_x(z):=
\begin{cases}
\displaystyle
\frac{1}{d_x^{\HH}}
\sum_{\substack{e\in E\\ e\supseteq\{x,z\}}}
\frac{1}{|e|-1},& z\sim_{\HH}x,\\[3mm]
0,&\text{otherwise},
\end{cases}
\]
we have $\mu_x^{\alpha,\HH}=\alpha\delta_x+(1-\alpha)q_x$, and similarly for $y$.  Hence
\begin{align*}
\Wass^{\HH}(\mu_x^{\alpha,\HH},\mu_y^{\alpha,\HH})
&=\sup_{f\in 1\text{-}\Lip(V)}\sum_{z\in V}f(z)\left(\mu_{x}^{\alpha,\HH}(z)-\mu_{y}^{\alpha,\HH}(z)\right)\\
&=
\max_{j\in\{-1,0,1\}}
\sup_{f\in\mathcal A_j}
\left[
\alpha j+(1-\alpha)
\sum_{z\in V}f(z)(q_x(z)-q_y(z))
\right]\\
&=
\max_{j\in\{-1,0,1\}}
\bigl\{\alpha j+(1-\alpha)c_j\bigr\},
\end{align*}
where
\[
   c_j:=\sup_{f\in\mathcal A_j}
   \sum_{z\in V}f(z)(q_x(z)-q_y(z)).
\]
The set $\supp(q_x)\cup\supp(q_y)$ is finite.  Moreover, for every $f\in\mathcal A_j$ and every $z\in\supp(q_x)\cup\supp(q_y)$,
\[
   |f(z)|=|f(z)-f(y)|\le d_{\HH}(z,y)\le2.
\]
Hence each $c_j$ is finite.  Thus the Wasserstein distance is the maximum of three affine functions of $\alpha$.  It is therefore convex and piecewise affine with at most three affine pieces.  Since $d_{\HH}(x,y)=1$,
\[
   \kappa_\alpha^{\HH}(x,y)
   =1-\Wass^{\HH}(\mu_x^{\alpha,\HH},\mu_y^{\alpha,\HH}),
\]
and consequently the idleness function is concave and piecewise affine with at most three affine pieces.

It remains to identify a universal affine interval adjacent to $1$.  Fix $\alpha>1/2$, and choose an optimal transport plan $\pi$ from $\mu_x^{\alpha,\HH}$ to $\mu_y^{\alpha,\HH}$ satisfying \cref{lem:maximal-diagonal}.  Since
\[
   \mu_y^{\alpha,\HH}(x)\le 1-\alpha<\alpha
   =\mu_x^{\alpha,\HH}(x),
\]
we have
\[
   \pi(x,x)=\mu_y^{\alpha,\HH}(x).
\]
Suppose that $\pi(x,y)=0$.  Then all surplus mass at $x$ must be transported to vertices outside $\{x,y\}$, so
\begin{align*}
\alpha-\mu_y^{\alpha,\HH}(x)
&=\sum_{v\notin\{x,y\}}\pi(x,v)\\
&\le \sum_{v\notin\{x,y\}}\mu_y^{\alpha,\HH}(v)\\
&=1-\alpha-\mu_y^{\alpha,\HH}(x),
\end{align*}
which implies $\alpha\le1/2$, a contradiction.  Therefore $\pi(x,y)>0$.  By \cref{lem:complementary-slackness}, every optimal Kantorovich potential $f$ for this $\alpha$ satisfies
\[
   f(x)-f(y)=1.
\]

Now fix $\alpha_0\in(1/2,1)$ and let $f_*$ be optimal at $\alpha_0$.  The preceding argument gives $f_*(x)-f_*(y)=1$.  At $\alpha=1$ the two measures are $\delta_x$ and $\delta_y$, and
\[
   \sum_{z\in V}f_*(z)(\delta_x-\delta_y)(z)
   =f_*(x)-f_*(y)
   =1
   =d_{\HH}(x,y)
   =\Wass^{\HH}(\delta_x,\delta_y).
\]
Hence $f_*$ is also optimal at $\alpha=1$.  By \cref{lem:common-potential-linearity}, the Wasserstein distance, and therefore the curvature, is affine on $[\alpha_0,1]$. Since $\alpha_0>1/2$ was arbitrary, the idleness function is affine on $(1/2,1]$, and continuity extends this to $[1/2,1]$.

Finally, $\kappa_1^{\HH}(x,y)=0$.  Hence there is a constant $C$ such that
\[
   \kappa_\alpha^{\HH}(x,y)=C(1-\alpha),
   \qquad \alpha\in[1/2,1].
\]
It follows that
\[
   \lim_{\alpha\uparrow1}
   \frac{\kappa_\alpha^{\HH}(x,y)}{1-\alpha}=C.
\]
By the definition of Lin--Lu--Yau curvature, $C=\kappa_{\LLY}^{\HH}(x,y)$, which proves \eqref{eq:limit-free-intro}.  Taking $\alpha=1/2$ gives \eqref{eq:half-to-lly}.
\end{proof}

\subsection{Uniform linear hypergraphs and the 2-section}

We recall the following result of Bourne et al. on idleness functions of graphs.

\begin{theorem}[Bourne et al.\ {\cite[Theorem~1.1]{BourneEtAl2018}}]\label{thm:piecewise linear property for graphs} Let $G=(V, E)$ be a locally finite graph. Let $x, y\in V$ with $x\sim_G y$. Then the function $\alpha\mapsto\kappa_{\alpha}^G(x, y)$ is concave and piecewise affine over $[0,1]$ with at most 3 affine parts. Furthermore $\kappa_{\alpha}^G(x, y)$ is affine on the intervals
$$\left[0,\frac{1}{\operatorname{lcm}\left(d_x^G, d_y^G\right)+1}\right]\text{ and }\left[\frac{1}{\max\left(d_x^G, d_y^G\right)+1}, 1\right].$$
Thus, if we have the further condition $d_{x}^G=d_{y}^G$, then $\kappa_{\alpha}^G(x, y)$ has at most two affine parts.
\end{theorem}

\begin{proof}[Proof of \cref{cor:uniform-linear-idleness}]
Assume that $\HH$ is $r$-uniform and linear.  Since every adjacent pair of vertices belongs to a unique hyperedge, \eqref{eq:lazy-measure} and \eqref{eq:2section-degree} give
\[
\mu_x^{\alpha,\HH}(z)
 =\mu_x^{\alpha,[\HH]_2}(z)=
 \begin{cases}
 \alpha,&z=x,\\[1mm]
 \displaystyle
 \frac{1-\alpha}{d_x^{[\HH]_2}},&z\sim_{[\HH]_2}x,\\[2mm]
 0,&\text{otherwise}.
 \end{cases}
\]
Thus the lazy EE random walk agrees exactly with the simple lazy random walk on $[\HH]_2$.  Since the graph distance in $[\HH]_2$ equals $d_{\HH}$,
\begin{equation}
\label{eq:2section-curvature}
   \kappa_\alpha^{\HH}(x,y)
   =\kappa_\alpha^{[\HH]_2}(x,y)
   \qquad(x\sim_{\HH}y).
\end{equation}
Passing to the high-idleness limit also gives
\[
   \kappa_{\LLY}^{\HH}(x,y)
   =\kappa_{\LLY}^{[\HH]_2}(x,y).
\]
\cref{eq:2section-degree,eq:2section-curvature} identify the idleness function with that of the corresponding edge in $[\HH]_2$. By \cref{thm:piecewise linear property for graphs}, the latter is affine on the corresponding intervals. Substituting
\[
   d_x^{[\HH]_2}=(r-1)d_x^{\HH},
   \qquad
   d_y^{[\HH]_2}=(r-1)d_y^{\HH}
\]
and using
\[
   \operatorname{lcm}\bigl((r-1)d_x^{\HH},(r-1)d_y^{\HH}\bigr)
   =(r-1)\operatorname{lcm}(d_x^{\HH},d_y^{\HH})
\]
gives the stated intervals.  If $d_x^{\HH}=d_y^{\HH}$, the two intervals meet, and hence the idleness function has at most two affine pieces.
\end{proof}

\subsection{Comparison with other random-walk models}
\label{subsec:other-walks}

The ORCHID framework \cite{CoupetteEtAl2023} includes the
equal-nodes (EN), equal-edges (EE), and weighted-edges (WE)
random walks.
For a nonisolated vertex $x$, write
\[
\begin{gathered}
 N_{\HH}(x):=\{y\in V:x\sim_{\HH}y\},\\
 c_{xy}:=\#\{e\in E:\{x,y\}\subseteq e\}
 \qquad(y\in N_{\HH}(x)).
\end{gathered}
\]
Their non-lazy transition probabilities are
\[
\begin{aligned}
 P_{\mathrm{EN}}(x,y)
   &:=\frac{1}{|N_{\HH}(x)|},\\
 P_{\mathrm{EE}}(x,y)
   &:=\frac{1}{d_x^{\HH}}
     \sum_{\substack{e\in E\\e\supseteq\{x,y\}}}
     \frac{1}{|e|-1},\\
 P_{\mathrm{WE}}(x,y)
   &:=\frac{c_{xy}}{\sum_{e\ni x}(|e|-1)}
 \qquad (y\in N_{\HH}(x)),
\end{aligned}
\]
and are zero otherwise.
For $s\in\{\mathrm{EN},\mathrm{EE},\mathrm{WE}\}$, put
\[
 \mu_x^{\alpha,s}
 :=\alpha\delta_x+(1-\alpha)P_s(x,\cdot),
 \qquad
 \kappa_\alpha^s(x,y)
 :=1-\frac{\Wass^{\HH}(\mu_x^{\alpha,s},\mu_y^{\alpha,s})}
              {d_{\HH}(x,y)},
\]
and
\[
\kappa_{\LLY}^{s}(x,y)
 :=
 \lim_{\alpha\uparrow 1}
 \frac{\kappa_\alpha^{s}(x,y)}{1-\alpha}.
\]
Note that the notation $\kappa_\alpha^{\HH}$ (resp., $\kappa_\LLY^{\HH}$) used elsewhere in this paper
is $\kappa_\alpha^{\mathrm{EE}}$ (resp., $\kappa_\LLY^{\mathrm{EE}}$).

\bigskip
\paragraph{\textbf{Idleness structure}.}
The proof of \cref{thm:idleness-main} applies to
\[
 \mu_x^\alpha=\alpha\delta_x+(1-\alpha)q_x
\]
for any probability measure $q_x$ supported on $N_{\HH}(x)$.
Indeed, the integer-potential argument uses only the
integer-valued metric and finite support.
For $\alpha>1/2$, the diagonal-saturation argument forces
$\pi(x,y)>0$ and hence $f(x)-f(y)=1$ for every optimal
Kantorovich potential, independently of the particular
transition probabilities.
Thus the three-piece conclusion and linearity on $[1/2,1]$
hold for each of the EN, EE, and WE walks.
In particular, their pairwise Lin--Lu--Yau limits exist and satisfy
\[
 \kappa_\alpha^s(x,y)
 =(1-\alpha)\kappa_{\LLY}^s(x,y),
 \qquad
 x\sim_{\HH}y,\quad \alpha\in[1/2,1).
\]
This common idleness structure does not imply equality of
the curvatures for different choices of $s$.

\bigskip
\paragraph{\textbf{Coincidence of the models}.}
If $\HH$ is $r$-uniform, then
\[
 P_{\mathrm{EE}}(x,y)
 =P_{\mathrm{WE}}(x,y)
 =\frac{c_{xy}}{(r-1)d_x^{\HH}}
 \qquad(x\sim_{\HH}y).
\]
If $\HH$ is linear, then $c_{xy}=1$ and
\[
 \sum_{e\ni x}(|e|-1)=|N_{\HH}(x)|,
\]
so $P_{\mathrm{WE}}=P_{\mathrm{EN}}$.
Consequently, all three walks coincide on $r$-uniform linear
hypergraphs, where
\[
 P_{\mathrm{EN}}(x,y)
 =P_{\mathrm{EE}}(x,y)
 =P_{\mathrm{WE}}(x,y)
 =\frac{1}{(r-1)d_x^{\HH}}
 \qquad(x\sim_{\HH}y).
\]
Their pairwise Ollivier--Ricci curvatures and Lin--Lu--Yau
curvatures therefore agree in this setting, without any
regularity assumption.

\bigskip
\paragraph{\textbf{Consequences for curvature assumptions}.}
For each fixed $s$ and $\alpha\in[1/2,1)$,
the pairwise curvatures $\kappa_\alpha^s$ and
$\kappa_{\LLY}^s$ have the same sign.
Hence, under the structural assumptions of
\cref{thm:linear-main}, its curvature hypothesis is equivalent to
\[
 \kappa_\alpha^s(x,y)\ge0
 \qquad\text{for every }x\sim_{\HH}y,
\]
for any fixed $s\in\{\mathrm{EN},\mathrm{EE},\mathrm{WE}\}$
and $\alpha\in[1/2,1)$.
For any fixed $\alpha\in[0,1)$, the same condition is still
sufficient: concavity and $\kappa_1^s(x,y)=0$ give
\[
 \kappa_{\LLY}^s(x,y)
 \ge \frac{\kappa_\alpha^s(x,y)}{1-\alpha}.
\]
The uniform nonlinear examples in \cref{thm:nonlinear-main}
also have positive WE Lin--Lu--Yau curvature, since EE and
WE agree on uniform hypergraphs.
For the linear nonuniform examples in
\cref{thm:linear-nonuniform}, their EE curvature computations
do not by themselves determine the EN or WE curvatures.
\section{A combinatorial inequality for cut hypergraphs}
\label{sec:comb}

Let $\HH=(V,E)$ be a bipartite hypergraph with respect to $V=X\sqcup Y$.  For adjacent $x\in X$ and $y\in Y$, recall that
\[
S_{xy}^{\HH}=S_{yx}^{\HH}
=
\{w\in Y\setminus\{y\}: w\sim_{\HH}x\}
\cup
\{w\in X\setminus\{x\}: w\sim_{\HH}y\}.
\]

The following combinatorial inequality is a key ingredient in the proof of \cref{thm:linear-main}, in the case where the cut hypergraph associated with a minimum edge cut is not a bipartite star.

\begin{theorem}
\label{thm:comb-main}
Let $\HH=(V,E)$ be a finite simple $r$-uniform linear hypergraph without isolated vertices, where $r\ge3$.  Suppose that $V=X\sqcup Y$ is a partition into nonempty parts, every hyperedge meets both $X$ and $Y$, and $\HH$ is not a bipartite star with respect to this bipartition.  Then
\begin{equation}
\label{eq:comb-main}
 \min_{\substack{(x,y)\in X\times Y\\x\sim_{\HH}y}}
 |S_{xy}^{\HH}|
 <
 (|E|+1)(r-1)-\frac{|V|}{2}-1.
\end{equation}
\end{theorem}
\begin{proof}
Let $\HH=(V,E)$ be a finite simple $r$-uniform linear hypergraph
without isolated vertices, where $r\ge 3$.  Suppose that $V=X\sqcup Y$ is a partition into nonempty parts $X$ and $Y$, every hyperedge meets both $X$ and $Y$, and $\HH$ is not a bipartite star with respect to $V=X\sqcup Y$.  For convenience, write
\[
 m:=|E|
 \qquad\text{and}\qquad
 n:=|V|.
\]

We first consider the case in which $\HH$ is connected.

\medskip
\noindent
\textbf{Case 1.}
For any edge
$e\in E$, we have
\[
    \min\{|e\cap X|,|e\cap Y|\}= 1.
\]
We consider the cross-neighbor graph $B(\HH)$ (see the definition in \cref{sec:edge-cuts}). For each $e\in E$, the assumption
\[
   \min\{|e\cap X|,|e\cap Y|\}=1
\]
implies that $e$ contributes exactly $r-1$ edges to $B(\HH)$.
By linearity, distinct hyperedges contribute disjoint sets of
cross-edges. Hence
\[
   |V(B(\HH))|=n,
   \qquad
   |E(B(\HH))|=m(r-1).
\]
Moreover, $B(\HH)$ has no isolated vertices. Since it is bipartite,
it is not $K_3$. If $B(\HH)$ were a star, then all hyperedges of
$\HH$ would contain its center and $\HH$ would be a bipartite star,
contrary to the hypothesis. Thus by \cref{eq:sxy-crossgraph,eq:S_xyHH_S_xyB(HH)}
and \cite[Theorem~1.11]{LiuXia2026}, we have
\begin{align*}
     \min_{\substack{(x,y)\in X\times Y\\x\sim_{\HH}y}}
 |S_{xy}^{\HH}|&=\min_{\substack{(x,y)\in X\times Y\\x\sim_{B(\HH)}y}}
 |S_{xy}^{B(\HH)}|\leq
 |E(B(\HH))|-\frac {|V(B(\HH))|}2 \\
 &=m(r-1)-\frac n2<(m+1)(r-1)-\frac n2-1.
\end{align*}
Thus \eqref{eq:comb-main} holds in this case.

\medskip
\noindent
\textbf{Case 2.} There exists an edge
$e\in E$ such that
\[
   \min\{|e\cap X|,|e\cap Y|\}\ge 2.
\]
Fix an edge $e\in E$ satisfying
\[
   \min\{|e\cap X|,|e\cap Y|\}\ge 2,
\]
and write
\[
   X_e:=e\cap X,
   \qquad
   Y_e:=e\cap Y.
\]
Let
\[
   X_e=\{u_1,\ldots,u_p\},
   \qquad
   Y_e=\{v_1,\ldots,v_q\},
\]
where, after interchanging $X$ and $Y$ if necessary,
\[
   2\le p\le q.
\]
In particular, $u_i\sim_{\HH}v_i$ for every $1\le i\le p$.

We prove that at least one of the pairs $(u_i,v_i)$ satisfies
\eqref{eq:comb-main} by contradiction. Suppose, to the contrary, that
\begin{equation}
\label{eq:contradiction-lower-S}
 |S_{u_iv_i}^{\HH}|
 \ge
 (m+1)(r-1)-\frac n2-1
 \qquad
 (1\le i\le p).
\end{equation}
Summing \eqref{eq:contradiction-lower-S} over $i$ gives
\begin{equation}
\label{eq:sum-lower-S}
 \sum_{i=1}^{p}|S_{u_iv_i}^{\HH}|
 \ge
 p\left(
     (m+1)(r-1)-\frac n2-1
   \right).
\end{equation}

We next derive an upper bound for the same sum. Since $\HH$ is
linear, $e$ is the unique hyperedge containing both $u_i$ and
$v_i$. The vertices in
\[
   e\setminus\{u_i,v_i\}
\]
contribute $r-2$ elements to $S_{u_iv_i}^{\HH}$. Every hyperedge
through $u_i$ other than $e$ contributes at most $r-1$ further
neighbors of $u_i$, and every hyperedge through $v_i$ other than
$e$ contributes at most $r-1$ further neighbors of $v_i$.
Possible overlaps only decrease the cardinality of the union.
Consequently,
\begin{equation}
\label{eq:individual-upper-S}
 |S_{u_iv_i}^{\HH}|
 \le
 \left(
   d_{u_i}^{\HH}+d_{v_i}^{\HH}-2
 \right)(r-1)+(r-2).
\end{equation}
Note that,
\begin{equation}
\label{eq:sum-degree-excess}
 \sum_{i=1}^{p}
 \left(
   d_{u_i}^{\HH}+d_{v_i}^{\HH}-2
 \right)
 \le m-1.
\end{equation}
Summing \eqref{eq:individual-upper-S} and using
\eqref{eq:sum-degree-excess}, we obtain
\begin{equation}
\label{eq:sum-upper-S}
 \sum_{i=1}^{p}|S_{u_iv_i}^{\HH}|
 \le
 (m-1)(r-1)+p(r-2).
\end{equation}
Combining \eqref{eq:sum-lower-S} and
\eqref{eq:sum-upper-S} yields
\[
 p\left(
     (m+1)(r-1)-\frac n2-1
   \right)
 \le
 (m-1)(r-1)+p(r-2).
\]
Since
\[
   (m+1)(r-1)-\frac n2-1
   =
   m(r-1)-\frac n2+(r-2),
\]
the terms $p(r-2)$ cancel, and therefore
\begin{equation}
\label{eq:key-contradiction-inequality}
 p\left(
     m(r-1)-\frac n2
   \right)
 \le
 (m-1)(r-1).
\end{equation}

Since $\HH$ is connected, its incidence graph $\mathcal{B}(\HH)$ is connected. Hence, by \eqref{eq:circuit_rank} and $\beta(\mathcal{B}(\HH))\geq 0$, we have
\[
   \beta(\mathcal{B}(\HH))=mr-(n+m)+1\geq 0,
\]
or equivalently,
\begin{equation}
\label{eq:connected-vertex-bound}
   n\le m(r-1)+1.
\end{equation}
It follows from \eqref{eq:connected-vertex-bound} that
\[
   m(r-1)-\frac n2
   \ge
   \frac{m(r-1)-1}{2}.
\]
Because $p\ge 2$, the left-hand side of
\eqref{eq:key-contradiction-inequality} satisfies
\[
\begin{split}
 p\left(
     m(r-1)-\frac n2
   \right)
 &\ge
 \frac p2\bigl(m(r-1)-1\bigr)\\
 &\ge
 m(r-1)-1\\
 &>
 m(r-1)-(r-1)\\
 &=
 (m-1)(r-1),
\end{split}
\]
contradicting \eqref{eq:key-contradiction-inequality}.

\medskip
We now consider the disconnected case. Let
\[
   \HH_i=(V_i,E_i),
   \qquad 1\le i\le k,
\]
be the connected components of $\HH$, where $k\ge2$. Put
\[
 m_i:=|E_i|,
 \qquad
 n_i:=|V_i|.
\]
Since every component has no isolated vertices,
\[
 m_i r=\sum_{v\in V_i}d_v^{\HH_i}\ge n_i,
\]
and hence,
\begin{equation}
\label{eq:component-positive}
 m_i(r-1)-\frac{n_i}{2}
 \ge
 \frac{m_i(r-2)}{2}>0
 \qquad(1\le i\le k).
\end{equation}

Suppose first that at least one component, say $\HH_1$, is
not a bipartite star with respect to
\[
 V_1=(X\cap V_1)\sqcup(Y\cap V_1).
\]
Applying the connected case to $\HH_1$, we obtain adjacent
vertices $x\in X\cap V_1$ and $y\in Y\cap V_1$ such that
\begin{align*}
 |S_{xy}^{\HH}|
 &=
 |S_{xy}^{\HH_1}|\\
 &<
 (m_1+1)(r-1)-\frac{n_1}{2}-1.
\end{align*}
Using \eqref{eq:component-positive}, we further obtain
\begin{align*}
 |S_{xy}^{\HH}|
 &<
 (m_1+1)(r-1)-\frac{n_1}{2}-1
 +\sum_{i=2}^{k}
 \left(
 m_i(r-1)-\frac{n_i}{2}
 \right)\\
 &=
 (m+1)(r-1)-\frac n2-1.
\end{align*}
This proves the desired inequality.

It remains to consider the case in which every connected component
is a bipartite star. For each $i$, let $c_i$ denote its center.
Linearity implies that the sets
\[
 e\setminus\{c_i\},
 \qquad e\in E_i,
\]
are pairwise disjoint. Therefore
\[
 n_i=1+m_i(r-1).
\]
Choose an index $j$ such that
\[
 m_j=\min_{1\le i\le k}m_i.
\]
For any leaf $y$ of $\HH_j$, we have
\[
 |S_{c_jy}^{\HH}|
 =
 |S_{c_jy}^{\HH_j}|
 =
 m_j(r-1)-1.
\]
Moreover,
\[
 n=\sum_{i=1}^{k}n_i
   =k+m(r-1).
\]
Consequently,
\begin{align*}
 (m+1)(r-1)-\frac n2-1
 &=
 m(r-1)-\frac {k+m(r-1)}2+r-2
 \\
 &=
 \frac {m(r-1)-k}2+r-2\\
 &\ge
 \frac {km_j(r-1)-k}2+r-2\\
 &=\frac k2 |S_{c_jy}^{\HH}|+r-2> |S_{c_jy}^{\HH}|.
\end{align*}
This finishes the proof of \cref{thm:comb-main}.
\end{proof}

\begin{remark}
The inequality in \cref{thm:comb-main} is a hypergraph extension of the graph-theoretic estimate in \cite[Theorem~1.11]{LiuXia2026}. The strict inequality is specific to the range $r\ge3$.
Formally substituting $r=2$ into its right-hand side gives
\[
\min |S_{xy}^{\HH}|<|E|-\frac{|V|}{2},
\]
whereas the corresponding graph inequality for $G=(V,E)$ is
\[
\min |S_{xy}^G|\le |E|-\frac{|V|}{2}.
\]
\end{remark}
\section{Nonnegative curvature and edge-connectivity}
\label{sec:connectivity}

\subsection{Curvature estimates on an edge cut}
The cut estimates of \cite[Theorems~4.1 and~4.2]{LiuXia2026} admit the following uniform-linear hypergraph form.
\begin{theorem}[Curvature estimates on an edge cut]
\label{thm:estimate-kLLY-mincut}
Let $\HH=(V,E)$ be a locally finite connected simple
$r$-uniform linear hypergraph, where $r\ge 3$.  Let
$E_0\subseteq E$ be a finite edge cut, and let
\[
   \HH_0=(V_0,E_0),
   \qquad
   V_0:=\bigcup_{e\in E_0}e,
\]
be the edge-induced subhypergraph on $E_0$.  Choose a nontrivial
partition
\[
   V=X\sqcup Y
\]
such that
\[
   E_0=\partial_{\HH}(X)=\partial_{\HH}(Y),
\]
and put
\[
   X_0:=X\cap V_0,
   \qquad
   Y_0:=Y\cap V_0.
\]
Then, for every $x\in X_0$ and $y\in Y_0$ satisfying
$x\sim_{\HH_0}y$, one has
\begin{equation}
\label{eq:estimate-kLLY-mincut}
\begin{split}
 &(r-1)d_x^{\HH}d_y^{\HH}\,
   \kappa_{\LLY}^{\HH}(x,y)\\
 &\quad\le
 \max\{d_x^{\HH},d_y^{\HH}\}
 \left(
    |V_0|+2|S_{xy}^{\HH_0}|
 \right)
 +2\min\{d_x^{\HH},d_y^{\HH}\}
 -2(r-1)d_x^{\HH}d_y^{\HH}.
\end{split}
\end{equation}

Suppose, in addition, that $\HH_0$ is a bipartite star with
respect to
\[
   V_0=X_0\sqcup Y_0.
\]
Assume, without loss of generality, that its center is
$x\in X_0$, and let $y\in Y_0$ be a leaf.  Then
\begin{equation}
\label{eq:estimate-kLLY-star}
\begin{split}
 &(r-1)d_x^{\HH}d_y^{\HH}\,
   \kappa_{\LLY}^{\HH}(x,y)\\
 &\quad\le
 d_x^{\HH}\left(
       \alpha_{x,y}^{\HH}+2-(r-1)d_y^{\HH}
     \right)
 +d_y^{\HH}\left(
       2|S_{xy}^{\HH_0}|+2-(r-1)d_x^{\HH}
     \right),
\end{split}
\end{equation}
where
\[
   \alpha_{x,y}^{\HH}
   :=
   \bigl|
      \{w\in V:
          w\sim_{\HH}x
          \text{ and }
          w\sim_{\HH}y
       \}
   \bigr|.
\]
\end{theorem}

\begin{proof}
Let
\[
   G:=[\HH]_2
\]
be the $2$-section of $\HH$.  Since $\HH$ is $r$-uniform and
linear, by \eqref{eq:2section-degree}, for every $z\in V$,
\begin{equation}
\label{eq:degree-two-section-cut-proof}
   d_z^{G}=(r-1)d_z^{\HH}.
\end{equation}
By \cref{cor:uniform-linear-idleness}, we have
\begin{equation}
\label{eq:LLY-two-section-cut-proof}
   \kappa_{\LLY}^{\HH}(u,v)
   =
   \kappa_{\LLY}^{G}(u,v).
\end{equation}

Note that $\HH_0$ is bipartite with respect to $V_0=X_0\sqcup Y_0$. Let $G_0$ be the cross-neighbor graph $B(\HH_0)$ of $\HH_0$ (see the definition in \cref{sec:edge-cuts}). Thus $E(G_0)$ is an edge cut of $G$. Moreover, if $\HH_0$ is a bipartite star centered at
$x\in X_0$, then $X_0=\{x\}$ and $G_0$ is
also a star centered at $x$.

For every $x\in X_0$ and $y\in Y_0$ satisfying
$x\sim_{\HH}y$, we have $x\sim_{G}y$. Applying \cite[Theorem~4.1]{LiuXia2026} to $G$, we obtain
\begin{align*}
 d_x^Gd_y^G\kappa_{\LLY}^{G}(x,y)
 &\le
 \max\{d_x^G,d_y^G\}
 \left(
    |V_0|+2|S_{xy}^{G_0}|
 \right)\\
 &\quad
 +2\min\{d_x^G,d_y^G\}
 -2d_x^Gd_y^G.
\end{align*}
If $\HH_0$ is a bipartite star centered at $x$, then $G_0$
is a star centered at $x$, and every $y\in Y_0$ is a leaf.
In this case, \cite[Theorem~4.2]{LiuXia2026} gives
\begin{align*}
 d_x^Gd_y^G\kappa_{\LLY}^{G}(x,y)\le d_x^G\bigl(\alpha_{x,y}^{G}+2-d_y^G\bigr) +d_y^G
 \left(
    2|S_{xy}^{G_0}|+2-d_x^G
 \right),
\end{align*}
where
\[
   \alpha_{x,y}^{G}
   :=
   \bigl|
      \{w\in V:
          w\sim_{G}x
          \text{ and }
          w\sim_{G}y
       \}
   \bigr|.
\]
Using
\eqref{eq:degree-two-section-cut-proof},
\eqref{eq:LLY-two-section-cut-proof}, and
\eqref{eq:S_xyHH_S_xyB(HH)}, and then dividing by $r-1$,
we obtain \eqref{eq:estimate-kLLY-mincut} and \eqref{eq:estimate-kLLY-star}.

\end{proof}

\subsection{Proof of the Edge-connectivity rigidity theorem}

\begin{proof}[Proof of \cref{thm:linear-main}]
Let $\HH=(V,E)$ be a locally finite connected simple $r$-uniform linear
hypergraph, where $r\ge 3$.  By \eqref{eq:trivial-upper}, we have
\[
   \edgecon(\HH)\le\delta(\HH).
\]
It therefore remains to prove the reverse inequality.

Let $E_0$ be a minimum edge cut of $\HH$:
\[
   E_0=\partial_{\HH}(X)=\partial_{\HH}(Y),
   \qquad
   V=X\sqcup Y.
\]
Set
\[
   m:=|E_0|=\edgecon(\HH)
\]
and let
\[
   \HH_0=(V_0,E_0)
   \qquad
   V_0:=\bigcup_{e\in E_0}e
\]
be the edge-induced subhypergraph on $E_0$.  Since $E_0$ is
finite, $\HH_0$ is finite.  It is also
simple, $r$-uniform, linear, has no isolated vertices, and every
one of its hyperedges meets both
\[
   X_0:=X\cap V_0
   \qquad\text{and}\qquad
   Y_0:=Y\cap V_0.
\]

We distinguish two cases.

\medskip
\noindent
\textbf{Case 1: $\HH_0$ is not a bipartite star.}

By \cref{thm:comb-main}, there exist
$x\in X_0$ and $y\in Y_0$ with
$x\sim_{\HH_0}y$ such that
\begin{equation}
\label{eq:comb-used-linear-main}
   |S_{xy}^{\HH_0}|
   <
   (m+1)(r-1)-\frac{|V_0|}{2}-1.
\end{equation}
Assume without loss of generality that
\[
   d_x^{\HH}\ge d_y^{\HH}.
\]
By the nonnegativity assumption and
\cref{thm:estimate-kLLY-mincut},
\begin{align}
0
&\le
(r-1)d_x^{\HH}d_y^{\HH}\kappa_{\LLY}^{\HH}(x,y)
\notag\\
&\le
d_x^{\HH}
\left(
   |V_0|+2|S_{xy}^{\HH_0}|
\right)
+2d_y^{\HH}
-2(r-1)d_x^{\HH}d_y^{\HH}.
\label{eq:curvature-chain-nonstar}
\end{align}
From \eqref{eq:comb-used-linear-main},
\[
   |V_0|+2|S_{xy}^{\HH_0}|
   <
   2(m+1)(r-1)-2.
\]
Substituting this strict inequality into
\eqref{eq:curvature-chain-nonstar} gives
\begin{align*}
0
&<
2d_x^{\HH}(m+1)(r-1)
-2d_x^{\HH}+2d_y^{\HH}
-2(r-1)d_x^{\HH}d_y^{\HH}\\
&=
2(r-1)d_x^{\HH}(m+1-d_y^{\HH})
+2(d_y^{\HH}-d_x^{\HH})\\
&\le
2(r-1)d_x^{\HH}(m+1-\delta(\HH)).
\end{align*}
Consequently,
\[
   m+1-\delta(\HH)>0.
\]
Since $m$ and $\delta(\HH)$ are integers, we obtain
\[
   \edgecon(\HH)=m\ge\delta(\HH).
\]

\medskip
\noindent
\textbf{Case 2: $\HH_0$ is a bipartite star.}

After interchanging $X$ and $Y$ if necessary, we may assume that
$\HH_0$ has center $x\in X_0$ and
\[
   X_0=\{x\}.
\]
Every hyperedge of $E_0$ contains $x$.  Since $\HH_0$ is linear, any two
of these hyperedges intersect only at $x$.  Therefore
\[
   |Y_0|=m(r-1).
\]
For every $y\in Y_0$,
\[
   d_y^{\HH_0}=1
\]
and
\begin{equation}
\label{eq:S-star-size}
   |S_{xy}^{\HH_0}|
   =
   |Y_0|-1
   =
   m(r-1)-1.
\end{equation}

If $X=\{x\}$, then
\[
   E_0=\{e\in E:x\in e\},
\]
so
\[
   m=d_x^{\HH}\ge\delta(\HH),
\]
and there is nothing to prove.  We may therefore assume that
$X\setminus\{x\}\ne\varnothing$.

Move $x$ from $X$ to $Y$ and consider the nontrivial partition
\[
   V=(X\setminus\{x\})\sqcup(Y\cup\{x\}).
\]
Every hyperedge of $E_0$ becomes entirely contained in $Y\cup\{x\}$,
whereas every hyperedge incident with $x$ but not belonging to $E_0$
becomes a boundary edge.  Hence
\[
 \left|
 \partial_{\HH}(X\setminus\{x\})
 \right|
 =
 d_x^{\HH}-m.
\]
By the minimality of $E_0$,
\[
   m\le d_x^{\HH}-m,
\]
and therefore
\begin{equation}
\label{eq:center-degree-star}
   d_x^{\HH}\ge 2m.
\end{equation}
Since $E_0$ is a minimum edge cut of $\HH$ and $X_0=\{x\}$, for any $y\in Y_0$, we have
\begin{equation}
\label{eq:alpha-star-bound}
   \alpha_{x,y}^{\HH}
   \le |Y_0|-1
   =m(r-1)-1.
\end{equation}

Applying the nonnegativity assumption and the star estimate
\eqref{eq:estimate-kLLY-star}, and then using
\eqref{eq:S-star-size}, \eqref{eq:center-degree-star}, gives
\begin{align*}
0
&\le
(r-1)d_x^{\HH}d_y^{\HH}\kappa_{\LLY}^{\HH}(x,y)\\
&\le
d_x^{\HH}\left(
       \alpha_{x,y}^{\HH}+2-(r-1)d_y^{\HH}
     \right)
+d_y^{\HH}\left(
       2m(r-1)-(r-1)d_x^{\HH}
     \right)\\
&=
d_x^{\HH}\left(
       \alpha_{x,y}^{\HH}+2-(r-1)d_y^{\HH}
     \right)
+(r-1)d_y^{\HH}(2m-d_x^{\HH})\\
&\le
d_x^{\HH}\left(
       \alpha_{x,y}^{\HH}+2-(r-1)d_y^{\HH}
     \right).
\end{align*}
It follows that
\[
   (r-1)d_y^{\HH}\le\alpha_{x,y}^{\HH}+2.
\]
Together with \eqref{eq:alpha-star-bound}, this yields
\[
   (r-1)d_y^{\HH}
   \le
   m(r-1)+1.
\]
Since $r-1\ge 2$,
\[
   d_y^{\HH}
   \le
   m+\frac{1}{r-1}
   <m+1.
\]
As $d_y^{\HH}$ and $m$ are integers, we conclude that
\[
   m\ge d_y^{\HH}\ge\delta(\HH).
\]

Thus, in both cases,
\[
   \edgecon(\HH)=m\ge\delta(\HH).
\]
Combining this with
$\edgecon(\HH)\le\delta(\HH)$ proves
\[
   \edgecon(\HH)=\delta(\HH).
\]
\end{proof}
\section{Sharpness beyond the uniform linear setting}
\label{sec:sharpness}
The two structural assumptions in \cref{thm:linear-main} are sharp in complementary directions. The first family below retains uniformity and drops linearity; the second retains linearity and drops uniformity. In both cases positive Lin--Lu--Yau curvature allows an arbitrarily large gap between minimum incidence degree and edge-connectivity. The constructions in this section were developed before I became
aware of the work of Burgess, Luther, and Pike
\cite{BurgessLutherPike2024}, which I cited during revision.

\subsection{The nonlinear uniform family}

We first give the following key example for the proof of \cref{thm:nonlinear-main}:
\begin{example}
\label[example]{ex:6.1}
For every pair of integers
\[
   r\ge 3,\qquad t\ge 2,
\]
we define the hypergraph $\HH_{r,t}$ as follows.

Let $A$ and $C$ be disjoint $t(r-1)$-element sets, let $b$ be a new
vertex, and choose partitions
\[
   A=A_1\sqcup\cdots\sqcup A_t,
   \qquad
   C=C_1\sqcup\cdots\sqcup C_t,
\]
where
\[
   |A_i|=|C_i|=r-1
   \qquad (1\le i\le t).
\]
Define
\[
   V_{r,t}:=A\sqcup C\sqcup\{b\}
\]
and
\begin{equation}
\label{eq:Hrt-edges}
\begin{split}
   E_{r,t}
   &:=
   \binomset{A}{r}\cup\binomset{C}{r}\\
   &\quad\cup
   \bigl\{\{b\}\cup A_i:1\le i\le t\bigr\}
   \cup
   \bigl\{\{b\}\cup C_i:1\le i\le t\bigr\}.
\end{split}
\end{equation}
Write
\[
   \HH_{r,t}:=(V_{r,t},E_{r,t}).
\]
\end{example}

We are now ready to prove \cref{thm:nonlinear-main}.
\begin{proof}[Proof of \cref{thm:nonlinear-main}]
It suffices to verify that, for any $r\in\mathbb{Z}_{r\geq 3}$ and $t\in\mathbb{Z}_{t\geq 2}$, the hypergraph $\HH_{r,t}$ constructed in
\cref{ex:6.1} has all the asserted properties, namely, that $\HH_{r,t}$ is finite, connected, simple, nonlinear, $r$-uniform, and has positive Lin--Lu--Yau curvature. 
Let
\[
   D:=\binom{t(r-1)-1}{r-1},
   \qquad
   L:=\binom{t(r-1)-2}{r-2}.
\]
We shall repeatedly use the identity
\begin{equation}
\label{eq:nonlinear-DL-identity}
   (t(r-1)-1)L=(r-1)D.
\end{equation}

Every hyperedge of $\HH_{r,t}$ has cardinality $r$, and all
hyperedges in \eqref{eq:Hrt-edges} are distinct. Thus
$\HH_{r,t}$ is a finite simple $r$-uniform hypergraph. Since any distinct vertices $u,v\in V_{r,t}$ are connected, $\HH_{r,t}$ is connected. Moreover, since $r\ge 3$, two suitable hyperedges in
$\binomset{A}{r}$ intersect in $r-1\ge 2$ vertices. Therefore
$\HH_{r,t}$ is nonlinear.

We first determine the minimum degree. For every
$x\in A\cup C$, the vertex $x$ belongs to $D$ hyperedges of the
hyperedge set $\binomset{A}{r}\cup\binomset{C}{r}$ and to exactly one
hyperedge containing $b$. Hence
\[
   d_b^{\HH_{r,t}}=2t,
   \qquad
   d_x^{\HH_{r,t}}=D+1
   \quad (\forall\ x\in A\cup C).
\]
Since
\[
   D=\binom{t(r-1)-1}{r-1}
   \ge t(r-1)-1
   \ge 2t-1,
\]
we obtain
\[
   \delta(\HH_{r,t})=2t.
\]

We next determine the edge-connectivity. Let
\[
   E_A
   :=
   \bigl\{\{b\}\cup A_i:1\le i\le t\bigr\}.
\]
Deleting all $t$ hyperedges in $E_A$ separates $A$
from $C\cup\{b\}$. Hence
\begin{equation}
\label{eq:nonlinear-edgecon-upper}
   \edgecon(\HH_{r,t})\le t.
\end{equation}

Conversely, let $F\subseteq E_{r,t}$ with $|F|<t$. Since
$r\ge 3$ and $t\ge 2$,
\[
   L=\binom{t(r-1)-2}{r-2}
   \ge t(r-1)-2
   \ge t.
\]
Every two distinct vertices of $A$ lie together in exactly $L$
hyperedges of $\binomset{A}{r}$. Since $|F|<t\le L$, at least
one such hyperedge remains after deleting $F$. Hence the
subhypergraph induced by $A$ remains connected. The same
argument applies to $C$.

Furthermore, each of the two families
\[
   \bigl\{\{b\}\cup A_i:1\le i\le t\bigr\},
   \qquad
   \bigl\{\{b\}\cup C_i:1\le i\le t\bigr\}
\]
contains $t$ hyperedges. Since $|F|<t$, at least one hyperedge
from each family remains. Thus $b$ is still connected to both
$A$ and $C$, and hence $\HH_{r,t}-F$ is connected. Therefore $\edgecon(\HH_{r,t})\ge t$. Together with \eqref{eq:nonlinear-edgecon-upper}, we obtain
\begin{equation}
   \edgecon(\HH_{r,t})=t
   =\delta(\HH_{r,t})-t.
\end{equation}

It remains to prove the curvature assertions. For any vertices $x\sim_{\HH_{r,t}}y$, by \cref{thm:idleness-main}, we have $\kappa_{\LLY}^{\HH_{r,t}}(x,y)=2\kappa_{1/2}^{\HH_{r,t}}(x,y)$. It now suffices to show that $\kappa_{1/2}^{\HH_{r,t}}(x,y)>0$. For convenience, write \[ \mu_v:=\mu_v^{1/2,\HH_{r,t}}. \] 
Since $d_{\HH_{r,t}}(x,y)=1$, by \cref{eq:alpha-curvature}, it is enough to prove that \[ \Wass^{\HH_{r,t}}(\mu_x,\mu_y)<1. \] The adjacency relations in $\HH_{r,t}$ are as follows: every two distinct vertices of $A$ are adjacent, every two distinct vertices of $C$ are adjacent, the vertex $b$ is adjacent to every vertex of $A\cup C$, and there is no adjacency between $A$ and $C$. Thus, up to interchanging $A$ and $C$ and interchanging $x$ and $y$, there are only two cases. 

\medskip \noindent \textbf{Case 1: $x,y\in A$.} Both $\mu_x$ and $\mu_y$ are supported on $A\cup\{b\}$. Moreover, any two distinct vertices of $A\cup\{b\}$ are adjacent in $\HH_{r,t}$.  Consequently, 
\[ 
d_{\HH_{r,t}}(u,v) = \begin{cases} 0,&u=v,\\ 1,&u\ne v, \end{cases} \qquad u,v\in A\cup\{b\}. 
\] 
Let $\pi_*$ be an optimal transport plan from $\mu_x$ to $\mu_y$ satisfying 
\[ \pi_*(u,u) = \min\{\mu_x(u),\mu_y(u)\} \qquad (u\in V_{r,t}), 
\]
see \cref{lem:maximal-diagonal}. Since both measures are supported on $A\cup\{b\}$, every pair $(u,v)$ for which $\pi_*(u,v)>0$ belongs to $(A\cup\{b\})\times (A\cup\{b\})$. Therefore, 
\[ 
\begin{split} \Wass^{\HH_{r,t}}(\mu_x,\mu_y) &= \sum_{u,v\in V_{r,t}} \pi_*(u,v)d_{\HH_{r,t}}(u,v)\\ &= \sum_{\substack{u,v\in A\cup\{b\}\\u\ne v}} \pi_*(u,v)= 1-\sum_{u\in A\cup\{b\}}\pi_*(u,u). 
\end{split} 
\]
Since
\[ \pi_*(x,x) = \min\{\mu_x(x),\mu_y(x)\} >0. \] 
It follows that
\[ \begin{split} \Wass^{\HH_{r,t}}(\mu_x,\mu_y) <1. \end{split}
\]
By symmetry, the same conclusion holds when $x,y\in C$.

\medskip \noindent \textbf{Case 2: $x\in A$ and $y=b$.} Suppose that $x\in A_i$. For convenience, write 
\[ 
n=t(r-1). 
\] 
By \eqref{eq:lazy-measure}, we have
\[
 \mu_b(z)
 =
 \begin{cases}
 1/2,
   &z=b,\\[2mm]
 \frac 1{4n},
   &z\in A\cup C,\\[4mm]
 0,
   &\text{otherwise};
 \end{cases}
\]
and
\[
 \mu_x(z)
 =
 \begin{cases}
 \frac{1}{2(D+1)(r-1)},
   &z=b,\\[2mm]
 1/2,
   &z=x,\\[2mm]
 \frac {L+1}{2(D+1)(r-1)},
   &z\in A_i\setminus\{x\},\\[4mm]
 \frac {L}{2(D+1)(r-1)},
   &z\in A\setminus A_i,\\[4mm]
 0,
   &\text{otherwise}.
 \end{cases}
\]
We first observe that
\[ \mu_b(u)<\mu_x(u) \qquad\text{for every }u\in A. \]
Indeed, this is immediate for $u=x$, since
\[ \mu_x(x)=\frac12>\frac{1}{4n}=\mu_b(x). \]
If $u\in A\setminus\{x\}$, then
\[
   \mu_x(u)
   \ge
   \frac{L}{2(D+1)(r-1)}.
\]
Since $n=t(r-1)$, \eqref{eq:nonlinear-DL-identity} gives
\[
\frac{L}{2(D+1)(r-1)}
=
\frac{D}{2(D+1)(n-1)}
\ge
\frac{1}{4(n-1)}
>
\frac{1}{4n}
=
\mu_b(u),
\]
where we used $D\ge1$. By \cref{lem:maximal-diagonal}, let $\pi_*$ be an optimal transport plan from $\mu_b$ to $\mu_x$ satisfying 
\[ \pi_*(u,u) = \min\{\mu_b(u),\mu_x(u)\} \qquad (u\in V_{r,t}).
\]
Since $\mu_b(u)<\mu_x(u)$ for every $u\in A$, we have
\[ \pi_*(u,u)=\mu_b(u) \qquad (u\in A). \] 
Thus all the mass of $\mu_b$ supported on $A$ remains fixed and contributes zero transportation cost. At the vertex $b$, since $\mu_x(b)<\mu_b(b)$, 
\[ \pi_*(b,b)=\mu_x(b). \] Moreover, $\mu_x$ is supported on $A\cup\{b\}$, so all the mass leaving $b$ must be transported to vertices of $A$. Since $d_{\HH_{r,t}}(b,v)=1$ for every $v\in A$, it follows that 
\begin{equation} \label{eq:b-transport-cost} \begin{split} \sum_{v\in V_{r,t}} \pi_*(b,v)d_{\HH_{r,t}}(b,v) = \sum_{v\in A}\pi_*(b,v)= \mu_b(b)-\pi_*(b,b)= \mu_b(b)-\mu_x(b). \end{split} 
\end{equation} 
Next consider the mass supported on $C$. Since
\[ \pi_*(b,b)=\mu_x(b), \] 
the entire target mass at $b$ has already been supplied by the mass initially located at $b$. Hence
\[ \pi_*(u,b)=0 \qquad (\forall\ u\in C). \]
Also, $\mu_x(C):=\sum_{w\in C}\mu_x(w)=0$, so no mass can be transported into $C$. Therefore, all mass initially supported on $C$ must be transported to $A$. Since 
\[ d_{\HH_{r,t}}(u,v)=2 \qquad (\forall\ u\in C,\ \forall\ v\in A), \] 
we obtain
\begin{equation} \label{eq:C-transport-cost} \begin{split} \sum_{\substack{u\in C\\v\in V_{r,t}}} \pi_*(u,v)d_{\HH_{r,t}}(u,v) = 2\sum_{\substack{u\in C\\v\in A}}\pi_*(u,v)= 2\sum_{u\in C}\mu_b(u)= 2\mu_b(C).
\end{split} 
\end{equation}
Combining \eqref{eq:b-transport-cost} and \eqref{eq:C-transport-cost}, and recalling that the mass supported on $A$ has zero transportation cost, we obtain
\[ 
\begin{split} \Wass^{\HH_{r,t}}(\mu_b,\mu_x) &= \sum_{u,v\in V_{r,t}} \pi_*(u,v)d_{\HH_{r,t}}(u,v)= \mu_b(b)-\mu_x(b)+2\mu_b(C)\\ &= \frac12 - \frac{1}{2(D+1)(r-1)} + 2\cdot\frac14= 1-\frac{1}{2(D+1)(r-1)}<1.
\end{split} 
\] 

By symmetry, the same conclusion holds when $x\in C$ and $y=b$.
\end{proof}

\subsection{A linear nonuniform family}

We next show that uniformity is equally indispensable in
\cref{thm:linear-main}. More precisely, even within the class of finite
connected simple linear hypergraphs, positive Lin--Lu--Yau curvature
does not force the edge-connectivity to be close to the minimum
degree.

\begin{example}\label[example]{ex:6.2}
For every integer $t\geq 1$, put
\[
    s:=t+1.
\]
Let $A$ and $C$ be disjoint sets of cardinality $2s$, let $b$ be a
new vertex, and choose partitions
\[
    A=A_1\sqcup\cdots\sqcup A_s,
    \qquad
    C=C_1\sqcup\cdots\sqcup C_s,
\]
where
\[
    |A_i|=|C_i|=2
    \qquad (1\leq i\leq s).
\]
Define
\[
    V_t:=A\sqcup C\sqcup\{b\}.
\]

Let
\[
    E_A
    :=
    \bigl\{
        \{x,y\}:
        x\in A_i,\ y\in A_j,\ 1\leq i<j\leq s
    \bigr\},
\]
and similarly
\[
    E_C
    :=
    \bigl\{
        \{x,y\}:
        x\in C_i,\ y\in C_j,\ 1\leq i<j\leq s
    \bigr\}.
\]
Furthermore, set
\[
    F_A
    :=
    \bigl\{
        \{b\}\cup A_i:1\leq i\leq s
    \bigr\},
    \qquad
    F_C
    :=
    \bigl\{
        \{b\}\cup C_i:1\leq i\leq s
    \bigr\}.
\]
Finally, define
\[
    E_t
    :=
    E_A
    \cup
    E_C
    \cup
    F_A
    \cup
    F_C
\]
and write
\[
    \mathcal{L}_t:=(V_t,E_t).
\]
\end{example}

We now verify the properties of $\mathcal{L}_t$ and thereby prove
\cref{thm:linear-nonuniform}.

\begin{proof}[Proof of \cref{thm:linear-nonuniform}]
We now establish the claimed properties of the hypergraph
$\mathcal{L}_t$ introduced in \cref{ex:6.2}, for an arbitrary
$t\geq 1$.
Fix $t\geq 1$ and put $s=t+1$. We first observe that $\mathcal{L}_t$ is a finite connected simple
linear hypergraph. Indeed, every hyperedge in
$E_A\cup E_C$ has cardinality $2$, while every
hyperedge in $F_A\cup F_C$ has cardinality $3$.
Two distinct hyperedges belonging to
$E_A\cup E_C$ clearly intersect in at most one
vertex. Any two distinct hyperedges belonging to
$F_A\cup F_C$ intersect exactly in $\{b\}$. 

Finally, every hyperedge in $E_A$ joins vertices belonging
to two distinct sets $A_i$ and $A_j$, whereas the two vertices
different from $b$ in a hyperedge of $F_A$ belong to a
single set $A_k$. Hence a hyperedge in $E_A$ and a
hyperedge in $F_A$ intersect in at most one vertex.
The same argument applies to $C$. Consequently,
\[
    |e\cap f|\leq 1
    \qquad
    \text{for all distinct }e,f\in E_t,
\]
and therefore $\mathcal{L}_t$ is linear. The same observation also
shows that no $2$-element hyperedge is contained in a $3$-element
hyperedge, so $\mathcal{L}_t$ is simple. Since $A$ and $C$ are both connected and $b$ is adjacent to
every vertex of $A\cup C$, the hypergraph $\mathcal L_t$ is connected.

We next determine the minimum degree. The vertex $b$ belongs to the
$2s$ hyperedges in
$F_A\cup F_C$, and hence
\[
    d_b^{\mathcal{L}_t}=2s.
\]
Let $x\in A_i$. For every $j\neq i$, the vertex $x$ forms a
$2$-element hyperedge with each of the two vertices of $A_j$.
Therefore $x$ belongs to $2(s-1)$ hyperedges of $E_A$,
and it belongs additionally to the unique hyperedge
$\{b\}\cup A_i$. Hence
\[
    d_x^{\mathcal{L}_t}=2(s-1)+1=2s-1
    \qquad
    (x\in A).
\]
By symmetry,
\[
    d_x^{\mathcal{L}_t}=2s-1
    \qquad
    (x\in C).
\]
It follows that
\[
    \delta(\mathcal{L}_t)=2s-1=2t+1.
\]

We next determine the edge-connectivity. Deleting the $s$ hyperedges in $F_A$ separates $A$ from
$C\cup\{b\}$. Therefore
\begin{equation}\label{eq:linear-lambda-upper}
    \edgecon(\mathcal{L}_t)\leq s.
\end{equation}

Conversely, let $F\subseteq E_t$ satisfy $|F|<s$. Consider first the
$2$-uniform subhypergraph on $A$ with edge set $E_A$.
Its underlying graph is the complete $s$-partite graph
\[
    K_{\underbrace{2,\ldots,2}_{s\text{ times}}},
\]
which is $(2s-2)$-edge-connected. Since
\[
    |F|<s\leq 2s-2,
\]
the vertices of $A$ remain connected after deleting $F$. The same
argument applies to $C$.

Moreover, each of the families $F_A$ and $F_C$
contains exactly $s$ hyperedges. Since $|F|<s$, at least one
hyperedge from each family remains after deleting $F$. Thus $b$
remains connected to both $A$ and $C$. Consequently,
$\mathcal{L}_t-F$ is connected, and therefore
\begin{equation}\label{eq:linear-lambda-lower}
    \edgecon(\mathcal{L}_t)\geq s.
\end{equation}
Combining \eqref{eq:linear-lambda-upper} and
\eqref{eq:linear-lambda-lower}, we obtain
\begin{equation}\label{eq:linear-gap}
    \edgecon(\mathcal{L}_t)
    =
    s
    =
    t+1
    =
    \delta(\mathcal{L}_t)-t.
\end{equation}

It remains to prove the curvature assertion. Let
$x\sim_{\mathcal{L}_t}y$. By \cref{thm:idleness-main},
\[
    \kappa_{\LLY}^{\mathcal{L}_t}(x,y)
    =
    2\kappa_{1/2}^{\mathcal{L}_t}(x,y).
\]
Thus it suffices to prove
\[
    \kappa_{1/2}^{\mathcal{L}_t}(x,y)>0.
\]
For convenience, write
\[
    \mu_v:=\mu_v^{1/2,\mathcal{L}_t}.
\]
Since $d_{\mathcal{L}_t}(x,y)=1$, it is enough to show
\[
    \Wass^{\mathcal{L}_t}(\mu_x,\mu_y)<1.
\]

The adjacency relations in $\mathcal{L}_t$ have a particularly
simple form. Every two distinct vertices of $A$ are adjacent, every
two distinct vertices of $C$ are adjacent, the vertex $b$ is
adjacent to every vertex of $A\cup C$, and there is no adjacency
between $A$ and $C$. Hence, up to interchanging $A$ and $C$ and
interchanging $x$ and $y$, there are only two cases.

\medskip
\noindent
\textbf{Case 1: $x,y\in A$.} Both $\mu_x$ and $\mu_y$ are supported on $A\cup\{b\}$. Moreover,
any two distinct vertices of $A\cup\{b\}$ are adjacent in
$\mathcal{L}_t$. Thus
\[
    d_{\mathcal{L}_t}(u,v)
    =
    \begin{cases}
        0, & u=v,\\
        1, & u\neq v,
    \end{cases}
    \qquad
    u,v\in A\cup\{b\}.
\]
Since $x\sim_{\mathcal{L}_t}y$, we have
\[
    \mu_y(x)>0,
\]
while
\[
    \mu_x(x)=\frac12.
\]
Since $\mu_x(x)=1/2$ and $\mu_y(x)>0$, we may keep
$\min\{\mu_x(x),\mu_y(x)\}=\mu_y(x)>0$ units of mass fixed at $x$.
Since every two distinct vertices in $A\cup\{b\}$ are adjacent,
the remaining mass can be transported at distance at most $1$.
Therefore,
\[
\Wass^{\mathcal{L}_t}(\mu_x,\mu_y)
\leq
1-\mu_y(x)
<1.
\]
By symmetry, the same conclusion holds when $x,y\in C$.

\medskip
\noindent
\textbf{Case 2: $x\in A$ and $y=b$.} Suppose that $x\in A_i$. Recall that
\[
    d_b^{\mathcal{L}_t}=2s,
    \qquad
    d_x^{\mathcal{L}_t}=2s-1.
\]
By \cref{eq:lazy-measure}, we have
\[
    \mu_b(z)
    =
    \begin{cases}
        \dfrac12,
            & z=b,\\[6pt]
        \dfrac{1}{8s},
            & z\in A\cup C,\\[6pt]
        0,
            & \text{otherwise},
    \end{cases}
\]
and
\[
    \mu_x(z)
    =
    \begin{cases}
        \dfrac12,
            & z=x,\\[6pt]
        \dfrac{1}{4(2s-1)},
            & z=b,\\[6pt]
        \dfrac{1}{4(2s-1)},
            & z\in A_i\setminus\{x\},\\[6pt]
        \dfrac{1}{2(2s-1)},
            & z\in A\setminus A_i,\\[6pt]
        0,
            & \text{otherwise}.
    \end{cases}
\]

We first note that
\begin{equation}\label{eq:linear-mass-comparison}
    \mu_b(u)<\mu_x(u)
    \qquad
    \text{for every }u\in A.
\end{equation}
Indeed, for $u=x$,
\[
    \mu_x(x)=\frac12>\frac{1}{8s}=\mu_b(x).
\]
If $u\in A_i\setminus\{x\}$, then
\[
    \mu_x(u)
    =
    \frac{1}{4(2s-1)}
    >
    \frac{1}{8s}
    =
    \mu_b(u),
\]
while, for $u\in A\setminus A_i$,
\[
    \mu_x(u)
    =
    \frac{1}{2(2s-1)}
    >
    \frac{1}{8s}
    =
    \mu_b(u).
\]
Thus \eqref{eq:linear-mass-comparison} holds.

By \cref{lem:maximal-diagonal}, let $\pi_*$ be an optimal transport plan
from $\mu_b$ to $\mu_x$ satisfying
\[
    \pi_*(u,u)
    =
    \min\{\mu_b(u),\mu_x(u)\}
    \qquad
    (u\in V_t).
\]
It follows from \eqref{eq:linear-mass-comparison} that
\[
    \pi_*(u,u)=\mu_b(u)
    \qquad
    (u\in A).
\]
Since the first marginal of $\pi_*$ is $\mu_b$, this exhausts all the
mass initially located at each vertex of $A$. Hence
\[
    \pi_*(u,v)=0
    \qquad
    (u\in A,\ v\neq u),
\]
and the mass of $\mu_b$ supported on $A$ contributes no transportation
cost.

At $b$, we have
\[
    \mu_x(b)
    =
    \frac{1}{4(2s-1)}
    <
    \frac12
    =
    \mu_b(b),
\]
and hence
\[
    \pi_*(b,b)=\mu_x(b).
\]
Since the second marginal of $\pi_*$ is $\mu_x$, the entire target
mass at $b$ is therefore already supplied by the mass initially
located at $b$. Consequently,
\begin{equation}\label{eq:no-C-to-b}
    \pi_*(u,b)=0
    \qquad
    (u\neq b).
\end{equation}

Moreover, $\mu_x$ is supported on $A\cup\{b\}$. Thus no mass leaving
$b$ can be transported to $C$, and by \eqref{eq:no-C-to-b} any mass
leaving $b$ must be transported to $A$. Since
\[
    d_{\mathcal{L}_t}(b,v)=1
    \qquad
    (v\in A),
\]
we obtain
\begin{align}
    \sum_{v\in V_t}
    \pi_*(b,v)d_{\mathcal{L}_t}(b,v)
    &=
    \sum_{v\in A}\pi_*(b,v)
    \notag\\&=
    \mu_b(b)-\pi_*(b,b)
    =
    \mu_b(b)-\mu_x(b).
    \label{eq:linear-cost-b}
\end{align}

We next consider the mass initially supported on $C$. Since
\[
    \mu_x(C)=0,
\]
the second marginal condition gives
\[
    \pi_*(u,v)=0
    \qquad
    (v\in C,\ u\in V_t).
\]
In addition, \eqref{eq:no-C-to-b} shows that no mass from $C$ can be
transported to $b$. Therefore every unit of mass initially located in
$C$ must be transported to $A$.

For every $u\in C$ and $v\in A$, we have
\[
    d_{\mathcal{L}_t}(u,v)=2.
\]
It follows that
\begin{align}
    \sum_{\substack{u\in C\\ v\in V_t}}
    \pi_*(u,v)d_{\mathcal{L}_t}(u,v)
    &=
    2\sum_{\substack{u\in C\\ v\in A}}
    \pi_*(u,v)
    \notag\\
    &=
    2\sum_{u\in C}\mu_b(u)
    =
    2\mu_b(C).
    \label{eq:linear-cost-C}
\end{align}

Combining \eqref{eq:linear-cost-b} and
\eqref{eq:linear-cost-C}, together with the fact that the mass
initially supported on $A$ has zero transportation cost, yields
\begin{align}
    \Wass^{\mathcal{L}_t}(\mu_b,\mu_x)
    &=
    \sum_{u,v\in V_t}
    \pi_*(u,v)d_{\mathcal{L}_t}(u,v)
    =
    \mu_b(b)-\mu_x(b)+2\mu_b(C)
    \\
    &=
    \frac12
    -
    \frac{1}{4(2s-1)}
    +
    2\cdot\frac14
    =
    1-\frac{1}{4(2s-1)}
    <1.
\end{align}
Therefore $\kappa_{1/2}^{\mathcal{L}_t}(b,x)>0$, and hence $\kappa_{\LLY}^{\mathcal{L}_t}(b,x)>0$.
By symmetry, the same conclusion holds when $x\in C$ and $y=b$.

We have therefore proved that
\[
    \kappa_{\LLY}^{\mathcal{L}_t}(x,y)>0
    \qquad
    \text{for every }
    x\sim_{\mathcal{L}_t}y.
\]
Together with \eqref{eq:linear-gap}, this completes the proof.
\end{proof}
\section*{Acknowledgements}

This manuscript combines and substantially reorganizes material from
the author's preprints arXiv:2608.01970 and arXiv:2608.06029.

\section*{Declaration of generative AI use}

During the preparation of the precursor manuscripts and the present work, the author used ChatGPT (OpenAI) for language editing and limited exploratory mathematical assistance. In particular, it was used interactively in developing the constructions in \cref{ex:6.1,ex:6.2} and suggested the use of the circuit-rank formula in the connected Case~2 of \cref{thm:comb-main}. All AI-assisted suggestions and the resulting mathematical arguments were independently checked and verified by the author, who takes full responsibility for the content of the article.


\begin{thebibliography}{99}

\bibitem{AsoodehGaoEvans2018}
S.~Asoodeh, T.~Gao, and J.~A. Evans,
\newblock Curvature of hypergraphs via multi-marginal optimal transport,
\newblock in \emph{2018 IEEE Conference on Decision and Control},
IEEE, 2018, pp.~1180--1185.

\bibitem{BourneEtAl2018}
D.~P. Bourne, D.~Cushing, S.~Liu, F.~M\"unch, and N.~Peyerimhoff,
\newblock Ollivier--Ricci idleness functions of graphs,
\newblock \emph{SIAM J. Discrete Math.} \textbf{32} (2018), no.~2,
1408--1424.

\bibitem{BurgessLutherPike2024}
A.~C. Burgess, R.~D. Luther, and D.~A. Pike,
\newblock The edge-connectivity of vertex-transitive hypergraphs,
\newblock \emph{J. Graph Theory} \textbf{105} (2024), no.~2, 252--259.

\bibitem{ChenLiuYou2025}
K.~Chen, S.~Liu, and Z.~You,
\newblock Connectivity versus Lin--Lu--Yau curvature,
\newblock \emph{Int. Math. Res. Not. IMRN} (2025), no.~19,
Article rnaf303.

\bibitem{CoupetteEtAl2023}
C.~Coupette, S.~Dalleiger, and B.~Rieck,
\newblock Ollivier--Ricci curvature for hypergraphs: a unified framework,
\newblock in \emph{The Eleventh International Conference on Learning Representations}, 2023.

\bibitem{DankelmannMeierling2016}
P.~Dankelmann and D.~Meierling,
\newblock Maximally edge-connected hypergraphs,
\newblock \emph{Discrete Math.} \textbf{339} (2016), no.~1, 33--38.

\bibitem{Diestel2017}
R.~Diestel,
\newblock \emph{Graph Theory}, 5th ed.,
\newblock Springer, Berlin, 2017.

\bibitem{EidiJost2020}
M.~Eidi and J.~Jost,
\newblock Ollivier Ricci curvature of directed hypergraphs,
\newblock \emph{Sci. Rep.} \textbf{10} (2020), Article 12466.

\bibitem{HuangLiuXia2024}
X.~Huang, S.~Liu, and Q.~Xia,
\newblock Bounding the diameter and eigenvalues of amply regular graphs
via Lin--Lu--Yau curvature,
\newblock \emph{Combinatorica} \textbf{44} (2024), no.~6, 1177--1192.

\bibitem{IkedaEtAl2022}
M.~Ikeda, Y.~Kitabeppu, Y.~Takai, and T.~Uehara,
\newblock Coarse Ricci curvature on hypergraphs associated with
nonlinear Kantorovich difference,
\newblock arXiv:2102.00698v5, 2026, accepted for publication in
\emph{Osaka J. Math.}

\bibitem{LinLuYau2011}
Y.~Lin, L.~Lu, and S.-T.~Yau,
\newblock Ricci curvature of graphs,
\newblock \emph{Tohoku Math. J.} \textbf{63} (2011), no.~4, 605--627.

\bibitem{LiuLuoOu2026}
S.~Liu, X.~Luo, and J.~Ou,
\newblock On the size of linear $r$-uniform hypergraphs and sufficient
conditions for maximally edge-connected hypergraphs,
\newblock \emph{Discrete Appl. Math.} \textbf{386} (2026), 58--69.

\bibitem{LiuXia2026}
S.~Liu and Q.~Xia,
\newblock Edge-connectivity and non-negative Lin--Lu--Yau curvature,
\newblock arXiv:2508.20950v2, 2026.

\bibitem{LiuYang2026}
S.~Liu and Y.~Yang,
\newblock Equivalence of Lin--Lu--Yau curvature and
$1/2$-Ollivier curvature on weighted graphs,
\newblock arXiv:2608.05939, 2026.

\bibitem{Ollivier2009}
Y.~Ollivier,
\newblock Ricci curvature of Markov chains on metric spaces,
\newblock \emph{J. Funct. Anal.} \textbf{256} (2009), no.~3, 810--864.

\bibitem{TianZhao2025}
Y.~Tian and L.~Zhao,
\newblock Lin--Lu--Yau Ricci curvature on hypergraphs,
\newblock arXiv:2507.04109v3, 2026.

\bibitem{Villani2003}
C.~Villani,
\newblock \emph{Topics in Optimal Transportation},
\newblock Graduate Studies in Mathematics, vol.~58,
American Mathematical Society, Providence, RI, 2003.

\bibitem{XiaIdleness2026}
Q.~Xia,
\newblock Idleness functions for Ollivier--Ricci curvature on hypergraphs,
\newblock arXiv:2608.01970v1, 2026.

\bibitem{ZhaoMeng2018}
S.~Zhao and J.~Meng,
\newblock Sufficient conditions for hypergraphs to be maximally
edge-connected,
\newblock \emph{Appl. Math. Comput.} \textbf{333} (2018), 362--368.

\bibitem{ZhaoTianMeng2020}
S.~Zhao, Y.~Tian, and J.~Meng,
\newblock Degree sequence conditions for maximally edge-connected
and super edge-connected hypergraphs,
\newblock \emph{Graphs Combin.} \textbf{36} (2020), 1065--1078.

\end{thebibliography}
\end{document}